\documentclass{amsart}
\usepackage{amssymb,amsmath,amsthm,amsfonts,amsopn,url,bbm, hyperref, comment}
\usepackage{enumitem}
\theoremstyle{plain}
\newtheorem{thm}{Theorem}[section]

\newtheorem{prop}[thm]{Proposition}
\newtheorem{lemma}[thm]{Lemma}
\newtheorem{cor}[thm]{Corollary}

\theoremstyle{definition}
\newtheorem{defn}[thm]{Definition}
\newtheorem*{defn*}{Definition}
\newtheorem{example}[thm]{Example}
\newtheorem*{example*}{Example}

\newtheorem{rmk}[thm]{Remark}
\newcommand{\field}[1]{\mathbbm{#1}}

\newcommand{\N}{\field{N}}
\newcommand{\Q}{\field{Q}}

\newcommand{\Z}{\field{Z}}
\newcommand{\ideal}[1]{\mathfrak{#1}}
\newcommand{\m}{\ideal{m}}
\newcommand{\n}{\ideal{n}}
\newcommand{\p}{\ideal{p}}
\newcommand{\q}{\ideal{q}}
\newcommand{\ia}{\ideal{a}}

\newcommand{\cal}{\mathcal}
\newcommand{\cA}{\cal{A}}

\DeclareMathOperator{\Ass}{Ass}
\DeclareMathOperator{\wAss}{\widetilde{Ass}}

\DeclareMathOperator{\Spec}{Spec}
\DeclareMathOperator{\Supp}{Supp}

\DeclareMathOperator{\Hom}{Hom}
\DeclareMathOperator{\Tor}{Tor}

\newcommand{\ra}{\rightarrow}
\newcommand{\arrow}[1]{\stackrel{#1}{\rightarrow}}

\DeclareMathOperator{\sK}{sK}
\DeclareMathOperator{\K}{K}
\DeclareMathOperator{\im}{im}
\DeclareMathOperator{\ann}{ann}
\newcommand{\onto}{\twoheadrightarrow}

\renewcommand{\phi}{\varphi}

\author{Neil Epstein}
\address{Department of Mathematical Sciences \\ George Mason University \\ Fairfax, VA  22030}
\email{nepstei2@gmu.edu}

\author{Jay Shapiro}
\address{Department of Mathematical Sciences \\ George Mason University \\ Fairfax, VA  22030}
\email{jshapiro@gmu.edu}

\title{Strong Krull primes and flat modules}


\date{December 20, 2013}
\begin{document}
\begin{abstract}
There are several theorems describing the intricate relationship between flatness and associated primes over commutative Noetherian rings.  However, associated primes are known to act badly over non-Noetherian rings, so one needs a suitable replacement.  In this paper, we show that the behavior of strong Krull primes most closely resembles that of associated primes over a Noetherian ring.  We prove an analogue of a theorem of Epstein and Yao characterizing flat modules in terms of associated primes by replacing them with strong Krull primes.  Also, we partly generalize a classical equational theorem regarding flat base change and associated primes in Noetherian rings.  That is, when associated primes are replaced by strong Krull primes, we show containment in general and equality in many special cases.  One application is of interest over any Noetherian ring of prime characteristic.  We also give numerous examples to show that our results fail if other popular generalizations of associated primes are used in place of strong Krull primes.
\end{abstract}

\maketitle

\section{Introduction}

The theory of associated primes is an essential part of any introduction to commutative Noetherian rings, and remains an important tool, especially in the modern age where primary decompositions may be sometimes computed effectively.  However, much of the theory breaks down in the non-Noetherian case.  Specifically, there are many axiomatic characterizations of the primes associated to an ideal (or a module), all equivalent for ideals or finite modules over a Noetherian ring, but which are all non-equivalent over general (non-Noetherian) commutative rings.  Hence, whenever one wants to generalize a theorem about Noetherian rings involving associated primes to the non-Noetherian case, the question arises as to which generalization, if any, is the right one for the problem at hand.  The problem with the \emph{usual} notion of associated primes (primes that are the annihilators of single elements of a module) is that $\Ass M$ is often empty even when $M \neq 0$.

The most common generalization seems to be the \emph{weakly associated} primes, or \emph{weak Bourbaki} primes of an $R$-module $M$, denoted $\wAss_R M$.  Their popularity  stems in part from the influence of Bourbaki, where they are introduced in an exercise \cite[Chapter IV, Section 1, exercises 17-19]{Bour-CA}.  One problem with this notion is that it does not admit maximal elements (see Example~\ref{ex:nonmax}).  Another popular choice is the set of \emph{Krull} primes of a module, denoted $\K_R(M)$, introduced in \cite{Kr-ohneend}; see also \cite{FHO-primal}, where infinite ``primal'' decomposition of ideals is explored via Krull primes.  One problem with the notion of Krull primes is that they do not respect short exact sequences the way that associated primes do over Noetherian rings (see Example~\ref{ex:KsK}).  Another is that they differ from the associated primes even over Noetherian rings (see Remark~\ref{rmk:Yao}).

Here, we will instead concentrate on the \emph{strong Krull} primes, or $\sK_R(M)$.  It is unclear whether to attribute these to McDowell \cite{Mcd-unpub} (where they have the name we use here) or Hochster \cite{Ho-grade} (where they correspond locally to saying a module has ``true grade 0'').  These seem to work better than the other two alternatives above for purposes of base change and other homological methods, a claim that we justify.  For example, we generalize a characterization of flat modules over a Noetherian ring \cite{nmeYao-flat}, due to the first named author and Y. Yao, to the non-Noetherian case by replacing associated primes with strong Krull primes (see Theorem~\ref{thm:flat}).  We also show that the analogue would fail for weakly associated primes (see Example~\ref{ex:EYwAss}).

We also show that if $S$ is a ring extension of $R$, $L$ an $R$-module and $M$  an $R$-flat $S$-module, then
$\sK_S(L \otimes_R M) \supseteq \bigcup_{\p \in \sK_RL} \sK_S(M / \p M)$ (see Theorem~\ref{thm:oneway}).  While equality does not hold in general (see Example~\ref{ex:onewayonly}), we are able to show that with certain additional assumptions on $R$ and $L$, equality does hold (e.g. Theorems~\ref{thm:Noethfg} and \ref{thm:Noeth}).  We then present an application of this result to a class of non-Noetherian ring extensions that appears naturally in the study of prime characteristic Noetherian rings (see Theorem~\ref{thm:perfclosure} and Corollary~\ref{cor:regcharp}).  On the other hand, if one were to use Krull primes instead, neither containment is valid, as we show in counterexamples (see Remark~\ref{rmk:Kfail} and Example~\ref{ex:Kfail}).

The structure of the paper is as follows: In \S\ref{sec:basics}, we introduce the reader to various generalizations of associated primes.  There we establish some key properties of these prime sets, including known ones as well as a maximality result for strong Krull primes.  In \S\ref{sec:flat}, we characterize flat modules using strong Krull primes and torsion-freeness.  The contents of \S\ref{sec:bc} are already outlined in the previous paragraph.  In \S\ref{sec:Hom}, we explore the relationship between strong Krull primes and $\Hom$-sets, generalizing a result of McDowell to a base change situation.  Finally, in \S\ref{sec:ex} we give examples to show that even though strong Krull primes are a useful tool, the intuition gleaned from working with associated primes over a Noetherian ring can lead one astray when working with strong Krull primes.

Throughout the paper, we make extensive use of examples in order to show which results are sharp and why strong Krull primes fit our purposes as well as they do.

\section{Basics}\label{sec:basics}

In this paper we consider four generalizations of the Noetherian ring notion of associated prime.  We begin by collecting the properties known to be satisfied by each definition.  Then we give examples to show that some properties fail for some of these notions.  In doing so, we justify our choice to concentrate mainly on the one given in the paper's title.

\begin{defn}
Let $R$ be a commutative ring, $M$ an $R$-module. The four subsets $\Ass_RM$, $\wAss_RM$, $\K_R(M)$, and $\sK_R(M)$ of $\Spec R$ are defined as follows: For $\p \in \Spec R$, \begin{itemize}
\item We say that $\p$ is an \emph{associated prime} of $M$ (that is, $\p \in \Ass_RM$) if there is some $z\in M$ such that $\p = \ann z$.
\item We say that $\p$ is \emph{weakly associated} to $M$ (that is, $\p \in \wAss_RM$) if there is some $z \in M$ such that $\p$ is minimal over the ideal $\ann z$. (These are sometimes called \emph{weak Bourbaki primes}.)
\item We say that $\p$ is a \emph{Krull prime} of $M$ (that is, $\p \in \K_R(M)$) if for any $a \in \p$, there is an element $z\in M$ such that $a \in \ann z \subseteq \p$.
\item We say that $\p$ is a \emph{strong Krull prime} of $M$ (that is, $\p \in \sK_R(M)$) if for any finitely generated ideal $I$ such that $I \subseteq \p$, there is an element $z\in M$ such that $I \subseteq \ann z \subseteq \p$. 
\end{itemize}
In all cases, the subscript $_R$ is optional when the ring is understood.  (Parentheses around the module in question are also optional.)
\end{defn}

Note that $\Ass_RM$,  $\wAss_RM$,  and $\sK_R(M)$ coincide whenever $R$ is Noetherian, and that $\K_R(M)$ also coincides when $R$ is Noetherian and $M$ is finitely generated.  In general, $\Ass (M) \subseteq \wAss(M) \subseteq \sK(M) \subseteq \K(M)$, with none of the implications reversible.  For an overview (including three additional non-equivalent generalizations), see \cite{IrRu-ass}.  For the reader's convenience, we provide a list of important properties in the next Proposition, most of which are covered in \cite{IrRu-ass}.  However:

\begin{rmk}\label{rmk:Yao}
It was incorrectly stated in \cite[p. 6]{Mcd-unpub} and reiterated in \cite[p. 346]{IrRu-ass} that  $\Ass_RM = \wAss_RM = \sK_R(M) =\K_R(M)$ for any $R$-module whenever $R$ is Noetherian.  However, there are Noetherian rings $R$ and $R$-modules $M$ for which $\K_R(M) \neq \Ass_RM$.

For instance, let $(R,\m)$ be any Noetherian local integrally closed domain of dimension at least two, and let \[
M := \bigoplus_{f \in \m} (R/fR) \cdot \mathbf{e}_f,
\]
where the $\mathbf{e}_f$ are free placeholder variables.  Then $\m \in \K_R(M)$, since for any $f$ in $\m$, $f \in \ann \mathbf{e}_f = fR \subset \m$.  However, $\m \notin \Ass_RM$.  Otherwise, there is some $z\in M$ such that $\m = \ann z$.  But then there is a finite set of elements $f_1, \dotsc, f_n \in \m$ such that $z=\sum_i \overline{a_i} \mathbf{e}_{f_i}$, so that $\m \in \Ass_R \left(\bigoplus_{i=1}^n (R / f_i R) \cdot \mathbf{e}_{f_i}\right)$.  But since associated primes respect direct sums, this means there is some $i$ such that $\m \in \Ass_R (R / f_i R)$. But since $R$ is normal, every associated prime of a  principal ideal has height at most 1, and since $\m$ has height at least two, we have our contradiction.

However, it is true that if $M$ is a \emph{finitely generated} module over a Noetherian ring $R$ (e.g. if it were cyclic, which for instance is the only case considered in \cite{FHO-primal}), then $\K_R(M) = \Ass_R(M)$, as can be seen via localization and prime avoidance.
\end{rmk}

\begin{prop}\label{pr:omnibus}
Let $\cA$ be one of $\Ass$, $\wAss$, $\K$, $\sK$.  Then we have the following, where $R$ is an arbitrary commutative ring with unity: \begin{enumerate}
\item For any $\cA$ above and $\p \in \Spec R$, $\cA_R(R/\p) = \{\p\}$.
\item Containment is respected: Namely, for any $\cA$ above, if $L \subseteq M$ is an $R$-submodule inclusion, then $\cA_R(L) \subseteq \cA_R(M)$.
\item The \emph{short exact sequence property}: namely, if $L \subseteq M$ is an $R$-submodule inclusion, then $\cA_R(M) \subseteq \cA_R(L) \cup \cA_R(M/L)$.  This property holds if $\cA = \Ass$, $\wAss$, or $\sK$.\footnote{But \emph{not} if $\cA=\K$; see Example~\ref{ex:KsK}.}
\item Detection of nonvanishing: If $M$ is an $R$-module, then $\cA_R(M) \neq \emptyset$ if and only if $M \neq 0$.  This property holds for $\cA = \wAss$, $\K$, or $\sK$ (but \emph{not} for $\cA=\Ass$\footnote{The usual counterexample is $R=k[X_1, X_2, \ldots] / (X_1^2, X_2^2, \ldots)$.}).
\item Local property: For $\p \in \Spec R$ and $M$ an $R$-module, we have $\p \in \cA_R(M)$ iff $\p R_\p \in \cA_{R_\p} (M_\p)$.  This property holds for $\cA = \wAss$, $\K$, or $\sK$ (but \emph{not} for $\cA=\Ass$\footnote{Let $k$ be a field and $R := k[X_1, X_2, \cdots, Y_1, Y_2, \cdots] / (\{X_1 X_n Y_n \mid n \in \N\})$.  Let $\p := (x_1, x_2, \cdots)$.  Then  $\p R_\p$ is the annihilator of $\frac{x_1}{1}$ in $R_\p$, but $0$ is the only element of $R$ that can be annihilated by all of $\p$; hence $\p \notin \Ass R$.}).
\item Generalized local property: For $W\subset R$ a multiplicative subset, we have \[
\cA_{R_W}(M_W) = \{\p R_W \mid \p \in \cA_R(M) \text{ and } \p \cap W = \emptyset\}.
\]
Again, this property holds for $\cA = \wAss$, $\K$, or $\sK$, but \emph{not} for $\cA=\Ass$.
\end{enumerate}
\end{prop}

At first blush, the notion of strong Krull prime looks more complicated than the others, and hence one might expect it to be too unwieldy to work with.  However, we have found that it is more robust than the other notions for our purposes.

\begin{example}\label{ex:KsK}
We present an example to show that unlike strong Krull primes, Krull primes do not obey the short exact sequence property.  The raw materials are more-or-less from McDowell \cite[Example 2.2]{Mcd-unpub}.

Let $D=F[\![x,y]\!]$, where $F$ is a field, and let $\m := (x,y)$ be the unique maximal ideal. Note that $y^s, x$ is a $D$-regular sequence for \emph{any} positive integer $s$.  Let $K$ be the total ring of quotients for $D$ (i.e. its field of fractions), and let $R := D + t K[\![t]\!]$, where $t$ is an analytic indeterminate over $K$.  Then $P := \m + t K[\![t]\!]$ is the unique maximal ideal of $R$.  Let $L := R/tR$.  As a $D$-module, this is just isomorphic to $D\oplus (K/D) t$ (with $t$ an indeterminate with respect to $D$), where the  $R$-module structure is obtained by having $t$ annihilate the module.   McDowell notes that $P$ is in $\K_R(L)$ but not in $\sK_R(L)$.  To see that $P$ is in $\K_R(L)$, one merely needs note that every element of $P$ is a zero-divisor on $L$.   However,  his  argument that $P \notin \sK_R(L)$ appears faulty; hence we give a different argument.  We will come at the result indirectly by producing a submodule $U$ of $L$ such that $P\notin \K_R(U) \bigcup \K_R(L/U)$. This actually proves two things; first, it shows Krull primes do not satisfy the short exact sequence property.  Second, since $\sK_R(N) \subseteq \K_R(N)$ for any $R$-module $N$, it shows that $P \notin \sK_R(U) \cup \sK_R(L/U)$, whence $P \notin \sK_R(L)$ (as strong Krull primes {\it do satisfy} the short exact sequence property; see Proposition~\ref{pr:omnibus}).

Let $U := (D_y / D) t$, considered as an $R$-submodule of $L$.  Then $P \notin \K_R(L/U)$ because $P$ contains an $(L/U)$-regular element -- namely $y$.  Similarly, $P \notin \K_R(U)$ because $P$ contains a $U$-regular element -- namely $x$ (for which we use the fact that $x$ is $(D/(y^s))$-regular for all positive integers $s$).
\end{example}

\begin{rmk}
 Note that in the above example we can replace $F[\![x,y]\!]$ with any local ring $D$ that contains a regular sequence of length 2, labelled say $y,x$.  In particular, one may set $D := F[x,y]_{(x,y)}$.  More generally, one may choose any integrally closed Noetherian local domain of dimension at least two.
\end{rmk}

Next, we provide a property that distinguishes the strong Krull primes from the weakly associated primes.

\begin{lemma}\label{lem:max}
Let $M$ be an $R$-module and $\p \in \sK_R(M)$.  Then there is some maximal element $P$ of  $\sK_R(M)$ such that $\p \subseteq P$.
\end{lemma}

\begin{proof}
Let $\{\p_j\}_{j \in J}$ be a chain in $\sK_R(M) \cap V(\p)$.  That is, $J$ is a totally ordered set, each $\p_j \in \sK_R(M)\cap V(\p)$, and for $j, j' \in J$ with $j \leq j'$, one has $\p_j \subseteq \p_{j'}$.  Let $Q := \bigcup_j \p_j$.  Let $I = (a_1, \dotsc, a_n)$ be a finitely generated subideal of $Q$.  There is some $j \in J$ such that $a_i \in \p_j$ for all $1\leq i \leq n$, and hence $I \subseteq \p_j$.  Since $\p_j \in \sK_R(M)$, there is some $z \in M$ such that $I \subseteq \ann z \subseteq \p_j \subseteq Q$, whence $Q \in \sK_R(M)$.  Since $Q \in V(\p)$ as well, we have shown that every chain in $\sK_R(M)\cap V(\p)$ has an upper bound.  Then by Zorn's lemma, $\sK_R(M)\cap V(\p)$ has a maximal element, which then must be a maximal element of $\sK_RM$ as well.
\end{proof}

\begin{example} \label{ex:nonmax}
The corresponding fact does \emph{not} hold for $\widetilde{\Ass}$.  Let $V$ be a valuation domain with value group $G:=\oplus_{i \in \N} \Z$ ordered lexicographically and valuation map $\nu$.  We introduce some notation.   Let $e_i\in G$ be the element with 1 in the $i$th coordinate and 0 everywhere else, so that the set $\{e_i\}_{i\in \N}$ is a basis for the free abelian group $G$.     For each $i=1,2 \ldots$ the set  \[
\p_i=\{a\in V \setminus \{0\} \mid \nu(a) = \sum_j n_j e_j >0 \text{ and} \min \{k \mid n_k \neq 0\} \leq i\} \cup \{0\}
\]
is a prime ideal of $V$.  Moreover these primes, along with the zero ideal and the maximal ideal $\m=\bigcup \p_i$, comprise all the prime ideals of $V$.   For each $i$, pick $a_i \in V$ such that $\nu(a_i) = e_i$; note that $\p_i$ is minimal over the ideal $(a_i)$.

Choose any $0\neq a \in \m$.  Then for any $\p_i$ such that $a\in \p_i$, we have that $\p_i \in \wAss_V(V/aV)$.  To see this, note that if $\nu(a) >e_i$, then $\frac{a}{a_i} \in V$ and $(aV :_V \frac{a}{a_i}) = (a_i)$; on the other hand if $\nu(a) \leq e_i$, then since $a \in \p_i$, we have $\nu(a) = e_i + n e_{i+1} + \sum_{j>i+1} n_j e_j$, where $n\leq 0$, in which case $\frac{a}{a_{i+1}} \in V$ and $(aV :_V a_{i+1}) = (a/a_{i+1})$.  In both cases, $\p_i$ is minimal over the resulting principal ideal.  

Next, note that $\m$ is not minimal over any proper subideal.  For suppose $\m$ is minimal over some ideal $J$.  Then for all $i$, $J \nsubseteq \p_i$.  Take any $x\in \m$.  Since $\m = \bigcup_i \p_i$, $x \in \p_i$ for some $i$.  Let $a \in J \setminus \p_i$.  Then $\nu(a)<\nu(x)$, so $x/a\in V$, and $x = a \cdot \frac{x}{a} \in J$, whence $J=\m$.  Since every ideal in $V$ of the form $(aV :_V b)$ is principal and $\m$ is not principal, it follows that $\m$ is not minimal over any such ideal, whence $\m \notin \wAss_V(V/aV)$.

On the other hand, $\m \in \sK_V(V/aV)$, as may be seen either by direct computation or from the fact that $\wAss_V(V/aV) \subseteq \sK_V(V/aV)$ along with Lemma~\ref{lem:max} and the fact that the only prime ideal containing the chain of elements comprising $\wAss_V(V/aV)$ is $\m$ itself.
\end{example}

The property from Lemma~\ref{lem:max} can be leveraged to provide the following local criterion for vanishing.

\begin{thm} \label{loczero}
Let $M$ be an $R$-module and $x\in M$.  The following are equivalent: \begin{enumerate}[label=(\alph*)]
\item\label{it:test0} $x=0$.
\item\label{it:testloc} $x/1 = 0$ in $M_\p$ for all $\p \in \sK_R(M)$.
\item\label{it:testmax} $x/1=0$ in $M_P$ for all \emph{maximal} elements $P$ of $\sK_R(M)$.
\end{enumerate}
\end{thm}

\begin{proof}
Clearly (\ref{it:test0}) $\implies$ (\ref{it:testloc}) $\implies$ (\ref{it:testmax}).  So assume $x \neq 0$.  Let $\p$ be a minimal prime over $\ann(x)$.  Then $\p \in \wAss(M) \subseteq \sK(M)$, hence there is some maximal element $P \in \sK(M)$ such that $\p \subseteq P$, by Lemma~\ref{lem:max}.  In particular, since $\ann(x) \subseteq P$, one has $x/1 \neq 0$ in $M_P$.
\end{proof}

\section{Flatness criteria}\label{sec:flat}

The next result, which gives many equivalent criteria for flatness of a module over a Noetherian ring, was proved by the first named author and Y. Yao in \cite{nmeYao-flat}.  For this, we say that a module $M$ over a commutative ring $R$ with total quotient ring $Q$ is \emph{torsion-free} if the natural map $M \ra M \otimes_RQ$ is injective.

\begin{thm}[{\cite[part of Theorem 2.2]{nmeYao-flat}}]\label{thm:EYflat}
Let $R$ be a commutative \textbf{Noetherian} ring, let $Q$ be the total quotient ring of $R$, and let $M$ be an $R$-module.  If $M \otimes_R Q$ is a flat $Q$-module, then the following are equivalent: \begin{enumerate}[label=(\alph*)]
\item $M$ is flat.
\item $\Ass_R(L \otimes_R M) \subseteq \Ass_RL$ for any $R$-module $L$.
\item $L \otimes_RM$ is torsion-free for every torsion-free $R$-module $L$.
\item $P \otimes_R M$ is torsion-free for every $P \in \Spec R$.
\item $\Tor_1^R(R/P, M)$ is torsion-free for every $P \in \Spec R$.
\end{enumerate}
In any case, the following are equivalent: \begin{enumerate}[label=(\roman*)]
\item $M$ is faithfully flat.
\item $M \otimes_RQ$ is flat over $Q$, and $\Ass_R(L \otimes_RM) = \Ass_RL$ for every $R$-module $L$.
\item $M$ is flat and $\Ass_R(L \otimes_RM) = \Ass_RL$ for every \emph{simple} module $L$.
\end{enumerate}
\end{thm}

Here we extend this characterization to the non-Noetherian case, using strong Krull primes in place of the Epstein-Yao usage of associated primes.   We start with the following lemma (where as in the rest of this paper, local does not necessarily mean Noetherian):

\begin{lemma} \label{lem:localcase} If $(R,\m)$ is local and $M$ is a flat $R$-module, then for any $R$-module $L$, if $\m \in \sK_R(L \otimes_R M)$, then $\m \in \sK_RL$.  The reverse implication holds if $M$ is faithfully flat.
\end{lemma}

\begin{proof}
First recall that for $(R,\m)$ local, if $N$ is any $R$-module, then $\m \in \sK_RN$ if and only if for every finitely generated proper ideal $I$ of $R$, $\Hom_R(R/I, N) \neq 0$:  If $0 \neq g \in \Hom_R(R/I,N)$, then $I \subseteq \ann_R(g(\bar1)) \subseteq \m$.  If conversely $I \subseteq \ann_Rz \subseteq \m$ for some $z\in N$, then $g: R/I \ra N$ given by $g(\bar{a}) := az$ is $R$-linear and nonzero.

Let $I$ be a finitely generated proper ideal of $R$.  Then since $R/I$ is a finitely presented $R$-module and $M$ is flat, a standard result of homological algebra shows that \[
\Hom_R(R/I, L \otimes_RM) \cong \Hom_R(R/I, L) \otimes_R M.
\]
But if the latter module is nonzero, then certainly $\Hom_R(R/I, L) \neq 0$.  This proves the first statement.  For the second statement, note that if $M$ is faithfully flat and $\Hom_R(R/I, L) \neq 0$, then this module tensored with $M$ is also nonzero.  Then the displayed isomorphism completes the proof.
\end{proof}

We are now ready to generalize Theorem~\ref{thm:EYflat} to the non-Noetherian case.

\begin{thm}\label{thm:flat}
Let $R$ be a commutative ring, let $Q$ be the total quotient ring of $R$, and let $M$ be an $R$-module.  If $M \otimes_R Q$ is a flat $Q$-module, then the following are equivalent: \begin{enumerate}[label=(\alph*)]
\item\label{it:EYflat} $M$ is flat.
\item\label{it:EYass} $\sK_R(L \otimes_R M) \subseteq \sK_RL$ for any $R$-module $L$,
\item\label{it:EYtf} $L \otimes_R M$ is torsion-free for every torsion-free $R$-module $L$.
\item\label{it:EYspec} $I \otimes_R M$ is torsion-free for every finitely generated ideal $I$ of $R$.
\item\label{it:EYtor} $\Tor_1^R(R/I, M)$ is torsion-free for every finitely generated ideal $I$ of $R$.
\end{enumerate}

In any case, the following are equivalent: \begin{enumerate}[label=(\roman*)]
\item\label{it:EYff} $M$ is faithfully flat.
\item\label{it:EYffass} $M \otimes_R Q$ is flat over $Q$, and $\sK_R(L \otimes_R M) = \sK_RL$ for every $R$-module $L$.
\item\label{it:EYsimple} $M$ is flat and $\sK_R(L \otimes_R M) = \sK_RL$ for every \emph{simple} $R$-module $L$.
\end{enumerate}
\end{thm}

\begin{proof}

First we prove (\ref{it:EYflat}) $\implies$ (\ref{it:EYass}).  Let $M$ be flat and $\p \in \sK_R(L \otimes_R M)$.  Then $M_\p$ is $R_\p$-flat and $\p R_\p \in \sK_{R_\p}(L_\p \otimes_{R_\p} M_\p)$, so by Lemma~\ref{lem:localcase}, we have $\p R_\p \in \sK_{R_\p} (L_\p)$, whence $\p \in \sK_RL$.

To see that (\ref{it:EYass}) $\implies$ (\ref{it:EYtf}), we assume (\ref{it:EYass}) and  prove the contrapositive of (\ref{it:EYtf}).  So suppose that $N := L \otimes_R M$ is not a torsion-free module.  Then there is some $R$-regular element $a$ and $0 \neq z \in N$ such that $az = 0$.  That is, $a \in \ann_Rz$.  Let $\p$ be a prime ideal minimal over $\ann z$.  Then $\p \in \wAss_R N \subseteq \sK_R N \subseteq \sK_R L$.  In particular, $a \in \p$ is a zero-divisor on $L$ (since every strong Krull prime of a module consists of zero-divisors on it), whence $L$ is not torsion-free.

The proofs of (\ref{it:EYtf}) $\implies$ (\ref{it:EYspec}) $\implies$ (\ref{it:EYtor}) may be copied over from the proof of \cite[Theorem 2.2]{nmeYao-flat} with barely any changes (replacing $P$ with $I$, of course).

To see that (\ref{it:EYtor}) $\implies$ (\ref{it:EYflat}), we may again follow the proof of the corresponding implication in the proof of \cite[Theorem 2.2]{nmeYao-flat}, this time noting the fact \cite[Theorem 7.8]{Mats} that the $R$-module $M$ is flat iff $\Tor_1^R(R/I,M)=0$ for every finitely generated ideal $I$ of $R$.

To see that (\ref{it:EYff}) $\implies$ (\ref{it:EYffass}), assume that $M$ is faithfully flat over $R$.  First note that standard base change arguments show that $M \otimes_R Q$ is flat over $Q$ and $M_\p$ is faithfully flat over $R_\p$ for all $\p \in \Spec R$.  Take any $\p \in \Spec R$.  Then \begin{align*}
\p \in \sK_RL  &\iff \p R_\p \in \sK_{R_\p} L_\p &\text{by Prop.~\ref{pr:omnibus} (6)}\\
&\iff \p R_\p \in \sK_{R_\p} (L_\p \otimes_{R_\p} M_\p) &\text{by Lemma~\ref{lem:localcase}}\\
&\iff \p \in \sK_R (L \otimes_R M) &\text{again by Prop.~\ref{pr:omnibus} (6)}.
\end{align*}

Then (\ref{it:EYffass}) $\implies$ (\ref{it:EYsimple})  because (\ref{it:EYass}) $\implies$ (\ref{it:EYflat}) when $M \otimes_RQ$ is flat over $Q$.

Finally, the proof that (\ref{it:EYsimple}) $\implies$ (\ref{it:EYff}) is the same as the corresponding implication in \cite[Theorem 2.2]{nmeYao-flat}, replacing $\Ass$ everywhere with $\sK$, and using the fact that for any $R$-module $N$, $N = 0 \iff \sK_R N = \emptyset$.
\end{proof}

\begin{rmk}
In addition to being a generalization of Theorem~\ref{thm:EYflat}, the above theorem is also a vast expansion of a result of Iroz and Rush.  Namely, if one combines Proposition 2.1 with Theorem 2.2 from \cite{IrRu-ass}, one obtains the following: Let $R$ be a commutative ring, $S$ a commutative flat $R$-algebra, and $L$ an $R$-module.  Then $\sK_R(L \otimes_R S) \subseteq \sK_RL$, with equality if $S$ is faithfully flat over $R$.

This represents, in the special case where $S=M$, the implications (\ref{it:EYflat}) $\implies$ (\ref{it:EYass}) and (\ref{it:EYff}) $\implies$ (\ref{it:EYffass}) of our theorem.
\end{rmk}

\begin{example}\label{ex:EYfg}
One might hope to be able to replace the finitely generated ideals in the implication (e) $\implies$ (a) with \emph{prime} ideals (i.e. transport Theorem~\ref{thm:EYflat} to the general case with no changes at all).  However, this cannot work in general.

Let $(R,\m,k)$ be any local ring such that $R \neq k$ and such that every prime ideal is idempotent, i.e. $P=P^2$ for all $P \in \Spec R$.  (For example, one can take $R$ to be a valuation domain with value group $\Q \oplus \cdots \oplus \Q$, any fixed number of copies, with lexicographic order.)  Let $M := R/\m$.  Let $P$ be any prime ideal of $R$.  Then \[
P \otimes_R M = P^2 \otimes_R R/\m \cong P \otimes_R P(R/\m) = P \otimes_R 0 = 0.
\]
But $\Tor_1^R(R/P, M)$ is a submodule of $P \otimes_R M=0$, hence $\Tor_1^R(R/P, M)=0$ as well, and $0$ is a torsion-free $R$-module.  Thus, condition (e) of Theorem~\ref{thm:EYflat} is satisfied.

However, $M$ cannot be flat over $R$.  To see this, suppose $M=k$ were flat over $R$.  Since $k \otimes_R \m$ is the kernel of the identity map $k=k \otimes_R R \ra k \otimes_R k=k$ (by flatness of $k$ applied to the short exact sequence $0 \ra \m \ra R \ra k \ra 0$), we have $k \otimes_R \m=0$.  On the other hand, for any nonunit $x\in R$, we have an injective  map $k \otimes_R(x)  \ra k \otimes_R \m$ (since $k$ is flat), whence $k \otimes_R (x) = (x)/(x)\m = 0$.  Then by the Nakayama lemma, $(x)=0$.  That is, all nonzero elements are units, whence $R$ is a field, contradicting our original assumption.  Thus, condition (a) of Theorem~\ref{thm:EYflat} fails.
\end{example}

We use the following convention: if $\phi: R \ra S$ is a ring homomorphism, we let $\phi^*: \Spec S \ra \Spec R$ be the corresponding map of prime spectra.

\begin{example}\label{ex:EYwAss}
The result of Theorem~\ref{thm:flat} does \emph{not} hold if $\sK$ is replaced by $\wAss$.  Indeed, the implication (a) $\implies$ (b) fails.

Heinzer and Ohm \cite[Example 4.4]{HeiO-locN} exhibit a flat extension $\phi: R \ra S$ of integral domains and an element $a \in R$ such that $\phi^*(\wAss_S(S/aS)) \nsubseteq \wAss_R(R/aR)$.  On the other hand, Lazard \cite[Proposition 3.1]{Laz-plat} shows that whenever $\phi:R \ra S$ is a flat ring map, $\wAss_RN = \phi^*(\wAss_SN)$ for any $S$-module $N$.  Hence, we have \[
\wAss_R((R/aR) \otimes_R S) = \wAss_R(S/aS) = \phi^*(\wAss_S(S/aS)) \nsubseteq \wAss_R(R/aR),
\]
providing failure of the implication (\ref{it:EYflat}) $\implies$ (\ref{it:EYass}) when $L=R/aR$ and $M=S$.
\end{example}

\section{Strong Krull primes and change of rings}\label{sec:bc}
 Let $\varphi: R\to S$ be a homomorphism of Noetherian rings.  Let $L$ be an $R$-module and $M$ an $S$-module that is flat as an $R$-module.   Then a useful theorem states that
  \begin{equation}\label{eq:a}
\Ass_S(L\otimes_R M) =\bigcup_{\p\in \Ass_R (L)}\Ass_S(M/\p M).
\end{equation}
  (see for example \cite[Theorem 23.2]{Mats}).
  
  In this section, we consider this result for non-Noetherian rings with strong Krull primes replacing associated primes.  That is, when $\varphi:R \ra S$ is a homomorphism of (not necessarily Noetherian) rings, $L$ is an $R$-module, and $M$ is an $R$-flat $S$-module, we ask: When does the equality \begin{equation}\label{eq:b}
  \sK_S(L \otimes_R M) = \bigcup_{\p \in \sK_RL} \sK_S(M/\p M)
  \end{equation}
  hold?  In the  general setting we prove the containment ``$\supseteq$''.  That is, $ \sK_S(L\otimes_R M) \supseteq \bigcup_{\p\in \sK_R (L)}\sK_S(M/\p M)$.  However, we also give an example to show that  equality need not hold.  We then show that the result holds if only $R$ is assumed to be Noetherian, so long as either $L$ is a finitely generated $R$-module or the ring map $R \ra S$ satisfies INC.  It is known that equality holds if $S =R[x]$ and $M=S$ \cite[Theorem 2.5]{IrRu-ass}.  Generalizing that result somewhat, we show that if $S$ is a \emph{content algebra} over $R$ (a notion recalled in Definition~\ref{def:content}), then the result holds if $M=S$ and $L$ is a cyclic $R$-module.   Finally, we show that if Krull primes are used instead of strong Krull primes, containment can fail in \emph{either} direction.

\begin{thm}\label{thm:oneway}
Let $\varphi: R \ra S$ be a ring homomorphism, $L$ an $R$-module, and $M$ an $R$-flat $S$-module.  Then \[
\sK_S(L \otimes_R M) \supseteq \bigcup_{\p \in \sK_RL} \sK_S(M / \p M).
\]
\end{thm}

\begin{proof}
Let $\p \in \sK_RL$ and $Q \in \sK_S(M/ \p M)$.  Let $J$ be a finitely generated ideal of $S$ contained in $Q$.  Then since $Q \in \sK_S(M/\p M)$, there is some $t \in M$ such that \[
J \subseteq (\p M :_S t) \subseteq Q.
\]
Then $J t$ is a finitely generated $S$-submodule of $\p M$. That is, we may write $Jt = \sum_{j=1}^n S a_j$, where each $a_j \in \p M$.  In particular, there are elements $m_i$ of $M$ and $p_{ij}$ of $\p$ such that \[
a_j = \sum_{i=1}^\ell p_{ij} m_i
\]
for each $1\leq j \leq n$.  Let $I := (p_{ij})$ be the ideal of $R$ generated by all the $p_{ij}$s.  Since $I$ is a finitely generated ideal contained in $\p$ and $\p \in \sK_RL$, there is some $u\in L$ such that \[
I \subseteq (0:_Ru) \subseteq \p.
\]
Moreover, $Jt \subseteq IM$.  So we have \[
Jt \subseteq IM \subseteq (0 :_R u)M \subseteq \p M.
\]
Taking the colon with $t$, we get \begin{equation}\label{eq:x}
J \subseteq (((0 :_R u)M) :_S t) \subseteq (\p M :_S t) \subseteq Q.
\end{equation}

\noindent \textbf{Claim:} $(((0 :_R u)M) :_S t) = (0 :_S (u \otimes t))$, where $u \otimes t$ is considered as an element of the $S$-module $L \otimes_R M$.

\vspace{3pt}

\noindent \textit{Proof of the claim.} 
Consider the following exact sequence of $R$-modules: \[
0 \ra (0 :_R u) \ra R \arrow{\cdot u} L.
\]
Since $M$ is flat over $R$, we may tensor over $R$ with $M$ to get the following exact sequence of $S$-modules: \[
0 \ra (0 :_R u)M \ra M \arrow{(u \otimes -)} L \otimes_R M.
\]
But the kernel $K$ of the rightmost map is the set of all $z \in M$ such that $u \otimes z = 0$ in $L \otimes_R M$.  So \[
(((0 :_R u)M) :_S t) = (K :_S t) = \{s \in S \mid (u \otimes st) = 0\} = (0 :_S (u \otimes t)),
\]
since $S$ acts on $L \otimes_R M$ via the second factor of the tensor product, finishing the proof of the claim.

\vspace{3pt}

Then (\ref{eq:x}) together with the claim above yields \[
J \subseteq (0 :_S (u \otimes t)) \subseteq Q,
\]
which finally implies that $Q \in \sK_R(L \otimes_R M)$, as was to be shown.
\end{proof}

We next present an example that shows that (\ref{eq:b}) can fail even when $L$ is cyclic and $M=S$.  (In another variation, Example~\ref{ex:onewaysamering} will show failure when $L=R=S$.)

\begin{example}\label{ex:onewayonly}
In this example we construct a ring extension $R\subset S$, with $S$  flat over $R$ (i.e.,  in (\ref{eq:b}), $S=M$), and  a cyclic $R$-module $L$  such that the containment in Theorem \ref{thm:oneway} is strict.
Specifically we will construct two valuation rings $R\subset S$ with value groups $\Q$ and $\Q \oplus \Q$ (with lex order), and valuations $\nu_1$ and $\nu_2$ respectively, such that for $a\in R$ we have $\nu_2(a) = (\nu_1(a),0)$.  Assuming that we have constructed such rings, let $\p$ denote the unique nonzero prime ideal of $R$ and $\q\subset \m$ the two nonzero prime ideals of $S$.   We claim that $\p S = \q$.  To see this, first note that clearly $\p S \subseteq \q$, since the first component of the value of any element in $\p S$ must be positive.  Conversely, for any $b \in \q$, one has $\nu_2(b) = (m,n)$ for some positive rational number $m$ and some $n \in \Q$.  Then choose $a\in \p$ with $\nu_1(a) = m/2$.  It follows that $\nu_2(b/a) = (m/2, n)$.  Hence $b/a \in \q$, and so $b = a (b/a) \in \p \q \subseteq \p S$.

Let $0\neq a\in \p$.  Then $\sK_R(R/aR)=\{\p\}$.   Note that since $R$ is a valuation ring and $S$ is torsion free over $R$, $S$ is flat over $R$.  Also, $R/aR\otimes_R S\cong S/aS$.  Similar to the construction in Example~\ref{ex:nonmax}, we can show that sK$_S(S/aS) =\{\q,\m\}$.   On the other hand, since $\p S =\q$, we have $\bigcup_{P \in \sK_R(R/aR)} \sK_S(S/PS)= \sK_S(S/\q) =\{\q\}$.

Now we construct $R$ and $S$.  Fix a field $K$.  Then the embedding of ordered groups  $\Q\to \Q\oplus \Q$, via $q\mapsto (q,0)$, induces an embedding of $F_1\hookrightarrow F_2$, where $F_1$ and $F_2$ are the quotient fields of the group rings $K[\Q]$ and $K[\Q\oplus \Q]$ respectively.  Now if $\nu_1$ and $\nu_2$ denote the natural valuations from $F_1$ and $F_2$ to $\Q$ and $\Q\oplus \Q$ respectively, the associated valuation rings $R$ and $S$ satisfy the required assumptions.
\end{example}

\begin{rmk}\label{rmk:Kfail}
It is clear from the definitions that if $A$ is a valuation ring (or any B\'ezout domain) then for any $A$-module $L$, one has $\K_A(L) = \sK_A(L)$.  Hence, the above example provides a situation where the containment ``$\subseteq$'' fails in Theorem~\ref{thm:oneway} when strong Krull primes are replaced with Krull primes.  Later, we will see that the containment ``$\supseteq$'' \emph{also} fails for Krull primes (cf. Example~\ref{ex:Kfail}).
\end{rmk}

We now present some special cases where (\ref{eq:b}) \emph{does} hold. First, we recall the following result of Iroz and Rush:

\begin{prop}\cite[Proposition 2.1]{IrRu-ass}\label{pr:pullback}
Let $\varphi: R \ra S$ be a ring homomorphism and let $N$ be an $S$-module.  Then \[
\varphi^*(\sK_SN) = \sK_RN.
\]
\end{prop}

Then note the following:
\begin{lemma}
\label{lem:MpM}
Let $R$ be a commutative ring, $M$ a flat $R$-module, and $\p \in \Spec R$.  If $M/\p M \neq 0$, then $\sK_R(M/\p M) = \Ass_R(M/\p M) = \{\p\}$.
\end{lemma}

\begin{proof}
First note that $M/ \p M \cong (R/\p) \otimes_R M$ is flat as an $(R/\p)$-module via base change to $R/\p$.
Moreover, any flat module over an integral domain is torsion-free, whence for any $0 \neq \bar z \in M/ \p M$, $\p = \ann_R(\bar z)$, which proves the result.
\end{proof}

We now have tools to prove a ``partially Noetherian'' case of equality in (\ref{eq:b}), as follows.

\begin{thm}\label{thm:Noethfg}
Let $\phi: R \ra S$ be a homomorphism of commutative rings.  Let $L$ be an $R$-module and $M$ an $R$-flat $S$-module.  Assume that $R$ is Noetherian and $L$ is finitely generated.
Then \[
\sK_S(L \otimes_R M) = \bigcup_{\p \in \Ass_RL} \sK_S(M/\p M).
\]
\end{thm}

\noindent \emph{Note:} Our proof is adapted from the proof of \cite[Theorem 23.2]{Mats}. 

\begin{proof}
By Theorem~\ref{thm:oneway}, we need only show ``$\subseteq$''.  Let $P \in \sK_S(L \otimes_R M)$.  Let \[
0 = Q_1 \cap \cdots \cap Q_r
\]
be a primary decomposition of the submodule $0 \subseteq L$ such that $\Ass L = \cup_i \Ass(L/Q_i)$.  Let $C_i := L/Q_i$.  Then there is an injective $R$-module map $L \hookrightarrow \bigoplus_{i=1}^r C_i$.  Hence, by flatness, we also have an injection \[
L \otimes_R M \hookrightarrow \bigoplus_{i=1}^r (C_i \otimes_R M).
\]
Then since strong Krull primes respect inclusion of modules and satisfy the strong exact sequence property, $P \in \sK_S(C_i \otimes_R M)$ for some $i$.  But by construction, $\Ass_R(C_i) = \{\p\}$ for some $\p \in \Ass_RL$.  We claim that $\phi^*(P) = \p$.

To see this, note that since $C_i$ is finitely generated and $\p$-coprimary, there is some positive integer $n$ with $\p^n C_i = 0$.  Then $\phi(\p^n)\cdot (C_i \otimes M) = 0$ as well, so that $\p^n \subseteq \phi^*(P)$ (since every strong Krull prime of a module contains the annihilator of that module), whence $\p \subseteq \phi^*(P)$.  On the other hand, for any $y \in R \setminus \p$, we get the following exact sequence: \[
0 \ra C_i \arrow{\cdot  y} C_i. \]
Then tensoring with the $R$-flat $S$-module $M$, we get the following exact sequence of $S$-modules: \[
0 \ra C_i \otimes_R M \arrow{\cdot \phi(y)} C_i \otimes_R M.
\]
Since every element of a strong Krull prime of a module is a zero-divisor on that module, we have $\phi(y) \notin P$, whence $y \notin \phi^*(P)$.  This shows that $\phi^*(P) \subseteq \p$, so that finally $\p =\phi^*(P)$, as per the claim.

Now, since $C_i$ is a finitely generated $R$-module, it has a prime filtration.  That is, there is a nested sequence of modules \[
0 = D_0 \subset D_1 \subset \cdots \subset D_k = C_i,
\]
and (not necessarily distinct) prime ideals $\p_1, \dotsc, \p_k \in \Spec R$ such that for each $j$, $D_j / D_{j-1} \cong R/\p_j$.

Tensoring with $M$ and using its flatness, we get $D_{j-1} \otimes_R M \subseteq D_j \otimes_R M$ and \[
(D_j \otimes_R M) / (D_{j-1} \otimes_R M) \cong M / \p_j M
\]
for each $1 \leq j \leq k$.

Hence by the short exact sequence property, \[
P \in \sK_S(C_i \otimes M) \subseteq \bigcup_j \sK_S(M/\p_j M).
\]
That is, there is some $j$ with $P \in \sK_S(M/ \p_j M)$.  But then \[
\p = \phi^*(P) \in \phi^*[\sK_S(M/\p_j M)] \subseteq \{\p_j\},
\]
where the last containment is a result of Lemma~\ref{lem:MpM} and Proposition~\ref{pr:pullback}. Thus, $\p = \p_j$, so $P \in \sK_S (M/ \p M)$ with $\p\in \Ass_RL$, as was to be shown.
\end{proof}

The above result is somewhat limited even in the Noetherian case, in that it assumes that $L$ is finitely generated.  There is, however, more that we can do:

\begin{thm}\label{thm:Noeth}
Let $\phi: R \ra S$ be a ring homomorphism, where $R$ is Noetherian.  Let $L$ be an $R$-module and $M$ an $R$-flat $S$-module.

\begin{enumerate}[label=(\roman*)]

\item Given any $P \in \sK_S(L \otimes_R M)$, there is some $P' \in (\phi^*)^{-1}(\phi^*(P))$ such that $P' \subseteq P$ and $P' \in \bigcup_{\p \in \Ass_RL} \sK_S(M/\p M)$,

\item The  sets $\sK_S(L \otimes_R M)$ and $\bigcup_{\p \in \Ass_RL} \sK_S(M/\p M)$ have the same \emph{minimal elements}, and

\item If $\phi$ further satisfies ``INC'' (that is, for any prime $\p \in \Spec R$, the elements of the set $(\phi^*)^{-1}(\{\p\})$ are pairwise incomparable), then \[
\sK_S(L \otimes_R M) = \bigcup_{\p \in \Ass_RL} \sK_S(M/\p M).
\]
\end{enumerate}
\end{thm}

\begin{proof}
In (i), let $\q := \phi^*(P)$. For the moment we will assume that $R$ is local with maximal ideal $\q$ and that $S$ is local with maximal ideal $P$.  By Proposition~\ref{pr:pullback}, we have $\q \in \sK_R(L \otimes_R M) = \Ass_R(L \otimes_RM)$ (since $R$ is Noetherian).  Thus, $\q$ is the $R$-annihilator of some element $\eta = \sum_{i=1}^t x_i \otimes y_i \in L \otimes_R M$, with each $x_i \in L$ and $y_i \in M$.  Let $L' := \sum_{i=1}^t R x_i \subseteq L$.  Then $\q$ is the $R$-annihilator of the element $\eta = \sum_{i=1}^t x_i \otimes y_i \in L' \otimes_RM$ (since $M$ is flat over $R$), so $\q \in \Ass_R(L' \otimes_R M) = \phi^*[\sK_S(L' \otimes_R M)]$.  Say $P' \in \sK_S(L' \otimes_R M)$ with $\q = \phi^*(P')$.  Then by Theorem~\ref{thm:Noethfg}, there is some $\p \in \Ass_RL' \subseteq \Ass_RL$ with $P' \in \sK_S(M/\p M)$.  Moreover, since $P$ is the unique maximal ideal of $S$, we have $P' \subseteq P$.

Now we drop the assumption that the rings are local with $\q$, $P$ as their respective maximal ideals.  Let $\psi: R_\q \ra S_P$ be the localized version of $\phi$. Since $P \in \sK_S(L \otimes_R M)$, we have $P S_P \in \sK_{S_P} ((L \otimes_R M)_P)$.  But $(L \otimes_R M)_P = L_\q \otimes_{R_\q} M_P$, so $P S_P \in \sK_{S_P} (L_\q \otimes_{R_\q} M_P)$.  Then the argument in the previous paragraph shows that there is some prime ideal $P' S_P \subseteq P S_P$ with $\psi^*(P S_P) = \psi^*(P' S_P) = \q R_\q$ and an element $\p R_\q \in \Ass_{R_\q} L_\q$ with $P'S_P \in \sK_{S_P}(M_P / \p M_P)$.  But this is the same as saying that there is some $\p \subseteq \q$ with $\p \in \Ass_R L$ and a prime ideal $P' \subseteq P$ with $\phi^*(P)= \phi^*(P') = \q$ such that $P' \in \sK_S (M / \p M)$.  This finishes the proof of (i).

Notice that (ii) follows directly from (i) and Theorem~\ref{thm:oneway}.

If moreover $\phi$ satisfies INC, then let $\q$, $P$, and $P'$ be as in the argument above.   Since $P' \subseteq P$ are members of the same antichain (namely $(\phi^*)^{-1}(\{\q\})$), we have $P = P'$, which (in conjunction with Theorem~\ref{thm:oneway}) shows that (iii) holds.
\end{proof}

One situation where this occurs is the following:  Let $R$ be a Noetherian reduced ring of prime characteristic $p>0$.  Let $R^\infty$ be the \emph{perfect closure} of $R$.  That is, $R^\infty$ results from adjoining unique $p^n$th roots of all the elements of $R$, for all $n\in \N$.  (This was first introduced in \cite{Gr-perfect}; it has proved quite useful in the theory of tight closure, e.g. \cite{HHmain, NoSh.alpha}\footnote{We adhere to the convention that authors be listed in alphabetical order.}.)  Then $R^\infty$ is also reduced.  However, $R^\infty$ is only Noetherian if $R$ is a product of finitely many fields \cite[Theorem 6.3]{NoSh.alpha}.

Recall that if a reduced Noetherian ring $R$ is excellent (e.g. any localization or completion of a quotient of a polynomial ring), then the regular locus  is an open, dense subset of $\Spec R$.  This means that there is a fixed nonzero ideal $I$ such that for a multiplicatively closed subset $W$ of $R$, $R_W$ is regular iff $I \cap W \neq 0$.  In this sense, ``most'' localizations of $R$ are regular.

Now, if $W \subseteq R$ is a multiplicative set such that $R_W$ is regular, then $(R_W)^\infty$ is flat over $R_W$ \cite{Kunz-regp}.  It is easily seen that $(R_W)^\infty\cong(R^\infty)_W$, so since $R_W$ is flat over $R$ and $(R_W)^\infty$ is flat over $R_W$, it follows that $(R_W)^\infty=(R^\infty)_W$ is an $R$-flat $R^\infty$-module.  Moreover, since $R^\infty$ is clearly integral over $R$, it satisfies ``INC'' as above.  Hence, we obtain the following application of Theorem~\ref{thm:Noeth}.

\begin{thm}\label{thm:perfclosure}
Let $R$ be a reduced Noetherian ring of prime characteristic $p>0$.  Let $W \subseteq R$ be a multiplicative set such that $R_W$ is regular.  Let $L$ be any $R$-module.  Then \[
\sK_{R^\infty}(L \otimes_R {(R^\infty)}_W) = \bigcup_{\p \in \Ass_RL} \sK_{R^\infty} ((R^\infty / \p R^\infty)_W).
\]
Hence, if $Q \in \Spec R^\infty$ with $Q \cap W = \emptyset$, then $Q \in \sK_{R^\infty} (L \otimes_R R^\infty)$ if and only if there is some $\p \in \Ass_RL$ such that $Q \in \sK_{R^\infty} (R^\infty / \p R^\infty)$.
\end{thm}

In the case where we can set $W = \{1\}$, we have the following.

\begin{cor}\label{cor:regcharp}
Let $R$ be a regular Noetherian ring of prime characteristic $p>0$.  Let $L$ be any $R$-module.  Then \[
\sK_{R^\infty} (L \otimes_R R^\infty) = \bigcup_{\q \in \Ass_RL} \sK_{R^\infty} (R^\infty / \q R^\infty).
\]
\end{cor}

We next show that equality holds in the case of cyclic $R$-modules,  when $S$ is a content algebra over $R$.  Recall the definition of \emph{content algebras}.
\begin{defn}\cite{OhmRu-content}\label{def:content}
Let $\phi: R \ra S$ be a homomorphism of commutative rings.  We define the \emph{content map} $c = c_\phi: S \ra \{$ideals of $R\}$ as follows: \[
c(f) := \bigcap\{I \mid I\text{ is an ideal of $R$ such that }f \in IS\}.
\]
We say that $S$ is a \emph{content algebra} over $R$ if the following four axioms hold: \begin{enumerate}
\item For all $f \in S$, $f \in c(f)S$.
\item For all $f \in S$ and $r\in R$, $c(rf) = rc(f)$.
\item $c(1)=R$.
\item (Dedekind-Mertens property) For all $f,g \in S$, there exists $n \in \N$ such that \[
c(f)^n c(fg) = c(f)^{n+1} c(g).
\]
\end{enumerate}
\end{defn}

A motivating example is where $S=R[X]$; in that case, one sees easily that the content of a polynomial is just the ideal generated by its coefficients.

We collect some known properties of content algebras below:
\begin{prop}\label{pr:content}
Let $\phi: R \ra S$ be such that $S$ is a content algebra over $R$. \begin{enumerate}[label=(\alph*)]
\item $S$ is faithfully flat as an $R$-module. \cite[Corollary 1.6]{OhmRu-content}
\item If $\p \in \Spec R$, then $\p S \in \Spec S$. \cite[Theorem 1.2]{Ru-content}
\item For all $f\in S$, $c(f)$ is a finitely generated ideal of $R$. \cite[discussion after 1.2]{OhmRu-content}
\item For any $f,g \in S$ such that $c(g)=R$, we have $c(fg) = c(f)$. \cite[6.1]{OhmRu-content}.
\item For any ideal $I$ of $R$ and any element $g\in S$, one has $g \in IS$ if and only if $c(g) \subseteq I$. \cite[1.2(iv)]{OhmRu-content}
\item\label{it:contentloc} Let $W \subseteq S$ be a multiplicative subset, and let $V := W \cap R$.  Suppose that for all $w\in W$, $c(w) \cap V \neq \emptyset$.  Then under the obvious localized map $\phi': R_V \ra S_W$, one has that $S_W$ is a content $R_V$-algebra, and the content map is given by $c(f/w) = c(f)_V$. \cite[Theorem 6.2]{OhmRu-content}
\item Hence if $W$ is the set of elements of $S$ of unit content, then $S_W$ is a content algebra over $R$ via the composite map $R \ra S \ra S_W$. \cite[6.3]{OhmRu-content}
\end{enumerate}
\end{prop}

We use the above properties without comment in proving the following result.

\begin{thm} \label{thm:content}
Let $R$ be a ring, $\ia$ an ideal of $R$, and $R \ra S$ a \emph{content algebra}.  Then \[
\sK_S(S/\ia S) = \{\p S \mid \p \in \sK_R(R/\ia)\} = \bigcup_{\p \in \sK_R(R/\ia)} \sK_S(S/\p S).
\]
\end{thm}

\begin{proof}
By Theorem \ref{thm:oneway} it suffices to show containment in only one direction, namely if $Q \in \sK_S(S/\ia S)$, then $Q = \p S$
for some $\p \in \sK_R(R/\ia)$.  We  will first show that $\p:=Q\cap R$ is in $\sK_R(R/\ia)$.  To that end, let $I$ be a finitely generated $R$-ideal contained in $\p$.  Then $IS$ is a finite $S$-ideal and $IS \subseteq \p S \subseteq Q$, so by our choice of $Q$, there is some $g\in S$ such that \[
IS \subseteq (\ia S :_S g) \subseteq Q.
\]
We have that $c(g)$ is a finitely generated ideal of $R$; say $c(g) = (r_1, \dotsc, r_n)$.  Then for any $a\in I$, we have $ag \in \ia S$, so that $a c(g) = c(ag) \subseteq \ia$, which means that $a r_i \in \ia$ for all $i$.  That is, \[
I \subseteq \bigcap_{i=1}^n (\ia :_R r_i).
\]
On the other hand, for any $a \in \bigcap_{i=1}^n (\ia :_R r_i)$, we have $c(ag) = a c(g) \subseteq \ia$, so that $ag \in \ia S$, whence $a \in ((\ia S :_S g) \cap R) \subseteq Q \cap R = \p$.  Thus, $\bigcap_{i=1}^n (\ia :_R r_i) \subseteq \p$, so there is some $j$ with $1 \leq j \leq n$ such that $(\ia :_R r_j) \subseteq \p$.  In sum, we have \[
I \subseteq (\ia :_R r_j) \subseteq \p,
\]
which shows that $\p \in \sK_R(R/\ia)$.

The last thing we need to show is that if $Q \in \sK_S(S/\ia S)$ and $\p = Q \cap R$, then $Q = \p S$.  Setting $W := R \setminus \p$, since  $S_W$ is a content $R_\p$-algebra, it is enough to show the result when $(R,\p)$ is local.

In that case, suppose there is some $b \in Q \setminus \p S$.  Then $c(b) \nsubseteq \p$, whence (since $\p$ is the unique maximal ideal of $R$), we have $c(b) = R$.  Since $Q \in \sK_S(S/\ia S)$, there is some $g \in S$ such that \[
b \in (\ia S :_Sg) \subseteq Q.
\]
But then $bg \in \ia S$, so that $c(g) = c(bg) \subseteq \ia$ (where the equality follows from the fact that $c(b) = R$), whence $g \in \ia S$ and $(\ia S :_Sg) = S$, contradicting the fact that this colon ideal is supposed to be in $Q$.
\end{proof}

In particular, the above result applies to the localization $R(X)$ of $R[X]$ at the set of polynomials with unit content, or (in view of Proposition~\ref{pr:content}, part \ref{it:contentloc}) many other localizations between $R[X]$ and $R(X)$.

\begin{example}\label{ex:Kfail}
We show here that the containment in Theorem~\ref{thm:oneway} does not hold for \emph{Krull} primes, even when $M=S$ is a polynomial extension and $L$ is cyclic.  Let $R$ be a ring, $L$ an $R$-module, and $P \in \Spec R$ such that $P \in \K_R(L) \setminus \sK_R(L)$ (recall the existence of such $R$, $L$, and $P$ from  Example~\ref{ex:KsK} -- an example which also shows that $L$ may be taken to be a cyclic $R$-module).  Let $u$ be an indeterminate over $R$ and $S := R[u]$.  Note that $L \otimes_R S \cong L[u]$.  Now  recall \cite[Theorem 2.5]{IrRu-ass}, which says that \[
\{\p R[u] \mid \p \in \sK_R(L)\} = \K_S(L[u]) = \sK_S(L[u]).
\]
It follows that although $P \in \K_R(L)$ (and clearly $\K_{R[u]}(R[u]/PR[u]) =\{PR[u]\}$, since $PR[u] \in \Spec R[u]$), one has $P R[u] \notin \K_{R[u]}(L[u])$.

Note also that the analogue of Proposition~\ref{pr:pullback} fails for Krull primes.  It is shown in \cite[Proposition 2.3]{IrRu-ass} that for any ring map $\phi: R \ra S$ and $S$-module $N$, one has $\phi^*[\K_S(N)] \subseteq \K_R(N)$.  However, letting $R$, $S$, $L$, $P$ be as above, set $N := L[u]$; then $P \in \K_R(N)$, but $P$ is not the contraction of any element of $\K_S(N)$.
\end{example}

\section{Strong Krull primes and $\Hom$-sets}\label{sec:Hom}
Next, we recall the following characterization of strong Krull primes in terms of $\Hom$-sets.  As the given reference is not widely available, we also provide a proof for the reader's convenience.

\begin{lemma}\cite[part of Proposition 1.4]{Mcd-unpub}\label{lem:fp}
Let $(R,\m)$ be a local commutative ring and $M$ an $R$-module.  Then $\m \in \sK_RM$ if and only if for every finitely presented nonzero $R$-module $L$, \[
\Hom_R(L,M) \neq 0.
\]
\end{lemma}

\begin{proof}
First, suppose $\Hom_R(L,M) \neq 0$ for all finitely presented nonzero $L$.  Then since for every finitely generated proper ideal $I$, $R/I$ is a finitely presented nonzero $R$-module, one has $\Hom_R(R/I, M) \neq 0$ for all such $I$.  But then if $0 \neq g \in \Hom_R(R/I, M)$ then $g(\bar{1}) \neq 0$, and we have $I \subseteq \ann g(\bar{1}) \subseteq \m$, where $\bar{1}$ is the image of $1$ in $R/I$.  Hence, $\m \in \sK_RM$.

Conversely, suppose $\m \in \sK_RM$.  Let $L$ be a finitely presented nonzero $R$-module.  This means that there is an exact sequence of the following form: \[
R^t \overset{\alpha = (r_{ij})}{\longrightarrow} R^s \rightarrow L \rightarrow 0.
\]
Here $(r_{ij})$ represents an $s \times t$ matrix of elements of $R$.  Since $R$ is local, we may assume that each $r_{ij} \in \m$.  (If not, one can represent some generator of $L$ as an $R$-linear combination of the other generators and shrink the matrix accordingly.)  Let $J$ be the ideal generated by the set $\{r_{1j} \mid 1 \leq j \leq t\}$.  Then $J$ is a proper, finitely generated ideal of $R$.  If we let $e_1, \dotsc, e_s$ be the unit basis vectors of $R^s$ with respect to the given matrix, let $U := J e_1 \oplus \left(\bigoplus_{i=2}^s R e_i\right)$ and let $K := \im \alpha$.  Then we claim that $K \subseteq U$.

To see this, note that $K$ is generated by the elements $k_j := \sum_{i=1}^s r_{ij} e_i$ for $1 \leq j \leq t$.  But $k_j = r_{1j} e_1 + (\sum_{i=2}^s r_{ij} e_i) \in U$ since $r_{1j} \in J$.  Therefore, there is a surjection $R^s / K \onto R^s  / U$.  But $R^s / U \cong R/J$.  Since $J$ is a finitely generated proper ideal and $\m \in \sK_R(M)$, we have $J \subseteq \ann z \subseteq \m$ for some $z \in M$, whence there is a nonzero $R$-linear map $\beta: R/J \rightarrow M$ given by $\bar{r} \mapsto rz$.  Hence, the composition \[
L \cong R^s / K \onto R^s / U \cong R/J \overset{\beta}{\rightarrow} M
\] 
is also a nonzero map, whence $\Hom_R(L, M) \neq 0$.
\end{proof}

The next result follows directly from the lemma.

\begin{thm}\cite[Proposition 2.6]{Mcd-unpub}\label{thm:McD-Hom}
Let $R$ be a commutative ring, $L$ a finitely presented $R$-module, and $M$ an $R$-module.  Then: \[
\sK_R \Hom_R(L,M) = \Supp_R L \cap \sK_R M.
\]
\end{thm}

Again, we provide a proof for the convenience of the reader.

\begin{proof}
Since membership in $\sK$ is a local property, we may assume $(R,\m)$ is local and then check for the presence of the element  $\m$ in each side of the equation.
  
First, assume $\m \in \sK_R \Hom_R(L,M)$.  Then for any finitely generated ideal $I$, we have \begin{align*}
0 \neq \Hom_R(R/I, \Hom_R(L,M)) &\cong \Hom_R(L \otimes_R R/I, M) \\
 &\cong \Hom_R(L, \Hom_R(R/I, M))
\end{align*}
by $\Hom$-$\otimes$ adjointness.

Thus, $L \neq 0$ (which means that $\m \in \Supp L$) and $\Hom_R(R/I, M) \neq 0$ (which shows that $\m \in \sK_R(M)$).

Conversely, assume that $\m \in \Supp_R L \cap \sK_RM$.  Say $L \cong R^s / K$, where $K$ is a finitely generated submodule of $R^s$.  Let $I$ be a finitely generated ideal of $R$. Then $L/IL \cong \frac{R^s}{K + I R^s}$.  But $K + I R^s$ is a finitely generated submodule of $R^s$, which means that $L/IL$ is finitely presented.  Moreover, it is nonzero by the Nakayama lemma (since $L \neq 0$ is finitely generated and $I \subseteq \m$).  Hence, by Lemma~\ref{lem:fp}, \[
0 \neq \Hom_R(L/IL, M) \cong \Hom_R(L \otimes_R R/I, M) \cong \Hom_R(R/I, \Hom_R(L,M)).
\]
Since $I$ was arbitrary, it follows that $\m \in \sK_R \Hom_R(L,M)$.
\end{proof}

Generalizing to a base change situation, we have the following:

\begin{thm}\label{thm:Hombc}
Let $\phi: R \ra S$ be a ring homomorphism, $L$ a finitely presented $R$-module, and $M$ an $S$-module.  Then \[
\sK_S \Hom_R(L,M) = (\phi^*)^{-1}(\Supp_R L) \cap \sK_S M.
\]
\end{thm}

\begin{proof}
We begin with the following claim:\[
\Supp_S(S \otimes_RL) = (\phi^*)^{-1}(\Supp_R L).
\]
\noindent\textit{Proof of the claim.}
It is enough to show that if $\gamma: (R,\m) \ra (S,\n)$ is a local homomorphism of (not necessarily Noetherian) local rings and $L$ a finitely presented $R$-module, then $L \neq 0$ if and only if $S \otimes_R L\neq 0$.  It is obvious that if $S\otimes_RL \neq 0$ then $L\neq 0$.  So suppose that $L \neq 0$.  Let $(a_{ij})$ be a minimal presenting matrix for $L$ over $R$.  Then each $a_{ij} \in \m$.  But then $(\gamma(a_{ij}))$ presents $S \otimes_R L$ over $S$ and each $\gamma(a_{ij}) \in \n$, hence the presentation is minimal, so that in particular, $S \otimes_R L \neq 0$.
\vspace{3pt}

Next note that $\Hom_R(L, M) \cong \Hom_S(S \otimes_R L, M)$ as $S$-modules, and that $S \otimes_RL$ is finitely presented over $S$ (just base-change the presenting matrix of $L$ over $R$ to $S$).  So by Theorem~\ref{thm:McD-Hom} above, we have \begin{align*}
\sK_S \Hom_R(L,M) &= \sK_S \Hom_S(S \otimes_R L, M) \\
&= \Supp_S(S \otimes_R L) \cap sK_SM \\
&=  (\phi^*)^{-1}(\Supp_R L) \cap \sK_S M,
\end{align*}
where the final equality follows from the initial claim.
\end{proof}

\section{Some cautionary counterexamples}\label{sec:ex}

The reader may by this time feel that strong Krull primes are a nearly perfect replacement in non-Noetherian rings to the notion of associated primes in Noetherian rings, and that most everything true in the latter context may be reformulated to be true in the former context.  In the interest of honing one's intuition, therefore, it is important to note the following examples where the behavior of strong Krull primes is really quite different from what one expects from working with associated primes over Noetherian rings. 

 \begin{example}
We first give an example of a finitely generated module $M$ over a ring $R$ where $\p \in \sK(M)$, but no submodule $L$ of $M$ satisfies $\sK (L) = \{\p\}$.

 Let $V$ be a valuation ring with value group $\Z\oplus \Z$ ordered lexicographically.  Thus $V$ has exactly two nonzero prime ideals $\p\subset \m$.  We first claim that for any $0\neq a\in \m$, sK$(V/aV)=\{\m\}$ if $a\not\in \p$ and  $\sK(V/aV)=\{\p,\m\}$, if $a\in \p$.   In the first case $\m$ is the only prime ideal that contains $aV$, so $aV$ is $\m$-primary. In the latter case, since $\p$ is minimal over $aV$, $\p \in$ sK$(V/aV)$.  Additionally,   for any $b\in \m\setminus aV$, $(aV: a/b) =bV$, and so if  $b\in \m\setminus \p$, a priori $\m$ is a minimal prime over $bV$.   Thus we also have  $\m\in $ sK$(V/aV)$, which completes the proof of the claim.

 Let $0\neq a \in \p$, so sK$(V/aV)=\{\p,\m\}$.  We will show that no submodule  $L$ of $V/aV$ satisfies sK$(L)=\{\p\}$.   It suffices to do this for cyclic submodules $L$.  However for $b\in \m\setminus aV$, $(aV:b) = (a/b)V$.  Thus $L$ must be isomorphic to $V/cV$ for some $c\in V$.  Hence by the claim, $\sK(L)\neq \{\p\}$.

This can never happen if $R$ is Noetherian, for in that case if $\p \in \sK(M) = \Ass(M)$, $R/\p$ embeds as an $R$-submodule of $M$, and $\{\p\} = \Ass_R(R/\p) = \sK_R(R/\p)$.

Note also that in this example, $\K_V(V/aV) = \wAss_V(V/aV) = \sK_V(V/aV)$, so that Krull primes and weakly associated primes exhibit the same pathology.
 \end{example}

\begin{example}\label{ex:axes} Recall that in a reduced Noetherian ring $R$, $\Ass(R)$ is exactly the set of  minimal primes of $R$.  We show below that a reduced ring $R$ may have the property that $\sK(R)$ includes a non-minimal prime, \emph{even when $R$ is local}.

Let $T=K[X_1,X_2,\ldots]_P$, where $P$ is the maximal ideal of $K[X_1,X_2,\ldots]$ generated by all the $X_i$, and $K$ is a field.   Let $R=T/(X_iX_j; i\neq j)$. Then $R$ is a reduced, non-Noetherian local ring of dimension 1. We denote the image of $X_i$ in $R$ by $x_i$.  The maximal ideal $\m$ of $R$ is generated by the set $\{x_i\}_{i=1,2, \ldots}$.    The other prime ideals are of the form $\p_i := (\{x_j \mid j \neq i\})$ (one for each positive integer $i$).     We will show that $\m \in \sK_R(R)$.  Let $I := (f_1, \dotsc, f_t)$ be a finitely generated ideal contained in $\m$.  Each $f_i$ can be written as a sum $g_{i1}+g_{i2}+\cdots+g_{im_i}$, where each $g_{ij}$ is a polynomial in a single variable $x_{s_{ij}}$, with constant term zero.  Pick an index $k$ such that $k>s_{ij}$ for all $i$ and $j$.   Then $x_kx_{s_{ij}} =0$ for all $i$. Thus $x_k f_i = 0$ for all $i$, whence $I \subseteq \ann(x_k) \subseteq \m$, so that $\m \in \sK_R(R)$.

Of course, $\m \in \K_R(R)$ as well, so Krull primes exhibit the same pathology.
\end{example}

\begin{example}\label{ex:onewaysamering}
If $\{M_i\}$ is an indexed set of $R$-modules, one always has $\Ass_R(\oplus_i M_i) = \bigcup_i \Ass_R(M_i)$.  Here we show that this equality does not hold when $\Ass$ is replaced by $\sK$.  Our construction also provides a counterexample to the equality (\ref{eq:b}) from \S\ref{sec:bc} in the case where $L=R=S$.

Let $V$, $\m$, and $\p_i$ ($i \in \N$) be as in Example~\ref{ex:nonmax}.  Fix an element $z\in V$ such that $v(z) = (1, 0, 0, \ldots)$, with zeros at all spots except the first entry.  Note that $z \in \p_1 \setminus (0)$.  Let $R := V/(z)$.  Then the prime spectrum of $R$ is $P_1 \subset P_2 \subset \cdots \subset \n$, where $P_i := \p_i / (z)$ and $\n := \m/(x)$.  For each positive integer $i$, let $M_i := R_{P_i}$, thought of as an $R$-module, and $M := \oplus_{i=1}^\infty M_i$.  First note that $\sK_R(M_n) = \sK_R(R_{P_n}) = \ell_n^*(\sK_{R_{P_n}} (R_{P_n})) \subseteq \ell_n^*(\Spec R_{P_n}) = \{P_1, \ldots P_n\}$ (with the second equality by Proposition~\ref{pr:pullback}), where $\ell_n: R \ra R_{P_n}$ is the localization map.

On the other hand, we have \begin{align*}
\sK_R(M_n) &= \sK_R(R_{P_n}) = \ell_n^*(\sK_{R_{P_n}}(R_{P_n})) \\
&= \ell_n^*(\{Q R_{P_n} \mid Q \in \sK_R(R), Q \subseteq P_n\}) \\
&= \{Q \in \sK_R(R) \mid Q \subseteq P_n\} \supseteq \{Q \in \wAss_R(R) \mid Q \subseteq P_n\} \\
&= \{\q/(z) \mid \q \in \wAss_V(V/(z)), \q \subseteq \p_n\} = \{\p_i / (z) \mid 1\leq i \leq n\} \\
&= \{P_i \mid 1 \leq i \leq n\}.
\end{align*}
Thus, $\sK_R(M_n) = \{P_1, \ldots, P_n\}$.

In particular, $\n \notin \bigcup_i \sK_R(M_i)$.  But $\n \in \sK_R(M)$.  To see this, let $I$ be a finitely generated  subideal of $\n$.  Then $I = yR$ for some $y\in \n = \bigcup_i P_i$, so that $y \in P_i$ (i.e. $I \subseteq P_i$) for some $i$.  But $P_i \in \wAss_R(M_i) \subseteq \sK_R(M_i)$, whence there is some $z\in M_i \subseteq M$ such that $I \subseteq \ann_R(z) \subseteq P_i \subseteq \n$.  Thus, $\sK$ does not respect countable direct sums.

Finally, we show that Equation~\ref{eq:b} fails for this $M$ when $L=R=S$.  We have $\n \in \sK_R(M)$.  However, for any $P \in \Spec R$, we have $\n \notin \sK_R(M/PM)$, as we now demonstrate.  When $P \neq \n$, then by Lemma~\ref{lem:MpM}, we have $\sK_R(M/PM) \subseteq \{P\}$, a set that does not contain $\n$.  On the other hand, $M/\n M = 0$.  To see this, let $z\in M$.  Then $z = \oplus_{i=1}^n \frac{r_i}{s_i}$ for some $n$, where each $\frac{r_i}{s_i} \in R_{P_i}$.  Let $a\in \n \setminus \bigcup_{i=1}^n P_i$.  Then $y = \oplus_{i=1}^n \frac{r_i}{as_i} \in M$ and $z=ay \in \n M$.  Since $M/\n M = 0$, we have $\sK_R(M/\n M) = \emptyset$, which again does not contain $\n$.
\end{example}

\begin{example}
Finally, in a Noetherian ring $R$, every ideal consisting of zero-divisors is contained in an element of $\Ass R$.  One might hope the analogous fact would hold for strong Krull primes of a  commutative ring, but such is not the case, \emph{even when $R$ is reduced}, as seen below.  Indeed, in our example, said ideal will even avoid containment in \emph{Krull} primes.

  Let $(D,Q)$ be a local Noetherian integral domain of positive dimension, and define $E$ to be the direct sum of denumerably many copies of $D/Q$,  with  multiplication on $E$ defined coordinatewise.   The bowtie ring $R:=D\bowtie E$ is defined by imposing the following multiplication on $D\oplus E$: if $d_1,d_2 \in D$ and $e_1,e_2 \in E$, then $(d_1,e_1) \cdot (d_2, e_2):=(d_1d_2, d_1e_2+d_2e_1+e_1e_2)$.   It was shown in \cite[Proposition 1]{DoSh-bowtie} that this ring is its own total ring of quotients, i.e., every nonunit of $R$ is a zero-divisor.  It is not difficult to see that for each (prime) ideal $P$ of $D$,  the set $P\bowtie E:= \{(d,a): d\in P\}$ is a (prime) ideal of $R$, which is a maximal ideal iff $P=Q$.  We write $\m :=  Q \bowtie E$.   Since $\m$ is a maximal ideal consisting of zero-divisors, to finish the example it will suffice to show that $\m\notin \K_R(R)$.  To this end, it is enough to show that $\m R_\m$ is not in $\K_{R_\m}(R_\m)$.  We first claim that $0\bowtie E$ is the kernel of the canonical map $R\to R_\m$.  To see this,  define for $i=1,2,\ldots$ elements $e_i:=(\delta_{ik})\in E$.  Then in $R$ we have $(1, - e_i)(0,  e_i)= (0,0)$.  Since $(1, - e_i)\in R\setminus \m$, it follows that $0\bowtie E$ is contained in the kernel.   As $0\bowtie E$ is a prime ideal, it in fact  must equal the kernel, proving the claim.  It then follows that $R_\m = D_Q =D$ and $\m R_\m = Q$.   Since $Q$ is not an associated prime of the Noetherian ring $D$, we have the desired example.
  
Of course, $\m$ also fails to be in $\wAss(R)$ or $\sK(R)$, so strong Krull primes and weakly associated primes exhibit the same pathology.
\end{example}

\section*{Acknowledgements}
The authors wish to thank Yongwei Yao for pointing out Remark~\ref{rmk:Yao} and for finding an error in an earlier draft.

\providecommand{\bysame}{\leavevmode\hbox to3em{\hrulefill}\thinspace}
\providecommand{\MR}{\relax\ifhmode\unskip\space\fi MR }
\providecommand{\MRhref}[2]{%
  \href{http://www.ams.org/mathscinet-getitem?mr=#1}{#2}
}
\providecommand{\href}[2]{#2}

\end{document}